\documentclass[11pt,a4paper]{article}
\usepackage[pagewise]{lineno} 
\usepackage[latin9]{inputenc}
\usepackage{amsmath}
\usepackage{amssymb}

\makeatletter

\title{\bf Malliavin Calculus and Stochastic Differential Equations }
\author{Shizan FANG$^1$\footnote{Corresponding author. Email: Shizan.Fang@u-bourgogne.fr;
{\bf 1.}Institut de Math\'ematiques de Bourgogne, UMR 5584 CNRS,
Universit\'e de Bourgogne Europe, F-21000 Dijon, France}
\quad and\quad Rongrong TIAN $^{1,2}$\footnote{Email: tianrr2018@whut.edu.cn; {\bf 2.}School of Mathematics and Statistics, Wuhan University of Technology, China.
The second author is supported by the China Scholarship Council (Grant No 202306950156)
and National Science Foundation of China (Grant No 11901442). }}

\vspace{3mm}

\usepackage{amsfonts}
\usepackage{amsthm}
\usepackage{array}
\usepackage{bm}
\usepackage{color}

\def\R{\mathbb{R}}
\def\N{\mathbb{N}}

\def\E{\mathbb{E}}
\def\P{\mathbb{P}}

\def\D{{\mathbb D}}

\def\M{\mathcal{M}}

\def\Id{\rm Id}

\def\eps{\varepsilon}
\def\<{\bigl<}
\def\>{\bigr>}

\let \dis=\displaystyle
\let\ra=\rightarrow

\newtheorem{theorem}{Theorem}[section]
\newtheorem{lemma}[theorem]{Lemma}\newtheorem{proposition}[theorem]{Proposition}\newtheorem{remark}[theorem]{Remark}

\makeatother

\begin{document}

\maketitle \makeatletter 
\global\long\def\theequation{\thesection.\arabic{equation}}
 \@addtoreset{equation}{section} \makeatother 


 \vskip 4mm
 \centerline{\bf In memory of Paul Malliavin on the occasion of his 100th Birthday}
 \vskip 8mm

\begin{abstract} This paper is devoted to a study on SDEs with a bounded Borel drift $b$. We first remark that the original integration by parts formula due to P. Malliavin \cite{Malliavin1} can be used
to deal with derivatives with respect to space variables, then we obtain a link between the product of heat kernels and iterated divergences in Malliavin calculus. An explicit estimate
for the derivative of solutions to SDE is obtained in terms of $||b||_\infty$; as a result, we prove that the SDE defines a continuous flow of maps in Sobolev spaces.
\end{abstract}

\vskip 4mm



\vskip 6mm







\section{Introduction}\label{sect1}

\quad In this work, we are interested in  the following stochastic differential equation (SDE) on $\R^d$ for $d\in \N^\ast$:

\begin{equation}\label{eq1.1}
dX_t=dW_t + b(X_t)\, dt,\quad X_0=x,
\end{equation}

where $b$, for simplicity of exposition,  is a time-independent bounded Borel function from $\R^d$ to $\R^d$ and $\{W_t; t\geq 0\}$ is a $d$-dimensional standard Brownian motion.
The well-posedness of \eqref{eq1.1} with bounded Borel functions was established in pioneer works \cite{Veretennikov, Zvonkin}.
It is also a classical result (see for example \cite{Kunita}) when $b$ is in  the space $C_b^{1,\alpha}$ of bounded functions of class $C^1$ with bounded derivatives which
are $\alpha$-H\"olderian,
SDE \eqref{eq1.1} defines  a flow of global diffeomorphisms $x\ra X_t(x)$ on $\R^d$; we denote its Jacobian matrix by $\dis Z_t(x)=\partial_xX_t(x)$.
Then $(Z_t)_{t\geq 0}$ solves the following random coefficient linear ordinary differential equation (ODE),

\begin{equation}\label{eq1.2}
dZ_t(x)=(\partial_xb)(X_t(x))\, Z_t(x)\, dt,\quad Z_0(x)=\Id.
\end{equation}

For a matrix $A$ in $\M_d(\R)$,  the space of $d\times d$ matrices of real numbers,
its Hilbert-Schmidt norm is denoted by $||A||_2$. Here is a first glance towards to estimating $Z_t(x)$,

\begin{equation*}
\begin{split}
d||Z_t(x)||_2^2 &= 2 \< Z_t(x), (\partial_x b)(X_t(x))Z_t(x)\>_2\, dt\\
&\leq 2 ||(\partial_x b)(X_t(x))||_2\ ||Z_t(x)||_2^2\, dt.
\end{split}
\end{equation*}
Gronwall lemma yields
\begin{equation*}
||Z_t(x)||_2 \leq \sqrt{d}\, \exp\Bigl(\int_0^t ||(\partial_x b)(X_s(x))||_2\, ds\Bigr).
\end{equation*}

So far,  for any $p\in [1, +\infty[$,
\begin{equation}\label{eq1.3}
\E\Bigl(||Z_t(x)||_2^p\Bigr) \leq d^{p/2}\, \E\Bigl[\exp\Bigl(p\int_0^t ||(\partial_x b)(X_s(x))||_2\, ds\Bigr)\Bigr].
\end{equation}

Therefore looking for moment estimates for $Z_t(x)$, a priori estimate of the right hand side of \eqref{eq1.3} in terms of $||b||_\infty$ plays a key role.
Neverthless it is impossible to utilize integration by parts formula to remove the derivative,
 since the norm $||\cdot||_2$ is involved there. We could consider partial derivatives $\partial_a X_t$ of $x\ra X_t(x)$  along  any fixed vector $a\in \R^d$, to have the ODE
on $\R^d$ :
\begin{equation}\label{eq1.4}
d(\partial_a X_t(x))=(\partial_xb)(X_t(x))\, \partial_a X_t(x)\, dt.
\end{equation}
But solving ODE \eqref{eq1.4} needs to consider the resolvent equation \eqref{eq1.2}.
\vskip 2mm

It is well-known (see \cite{FlandoliGP, ZhangX3}) that if $b$ has a little bit H\"older continuity, then $x\ra X_t(x)$ is a $C^1$-diffeomorphism of $\R^d$.
An important question raises: if the drift $b$ is only a bounded Borel function, is it possible that $X_t$ has Osgood continuity? Namely, for a constant $C>0$,
\begin{equation*}
||X_t(x)-X_t(y)||\leq C\, ||x-y||\ln{\frac{1}{||x-y||}},\quad\textup{for small}\quad ||x-y|| ?
\end{equation*}
 Many classical results remain true under Osgood continuity conditions,
see for example \cite{FangZ}.   Following \cite{Fang2},  if the $L^p$ norm of $\partial_xX_t(x)$ has at most linear growth with respect to $p$,
 then $X_t$ has Osgood continuity. Trying to give a response  motivates the investigation of this work.

\vskip 2mm
For $d=1$, denoting $X_t^\prime(x)$ the derivative with respect to the space variable $x$, we have

\begin{equation*}\label{eq1.5}
dX_t^\prime(x)=b^\prime(X_t(x))X_t^\prime(x)\, dt,
\end{equation*}
which implies 
\begin{equation}\label{eq1.6}
X_t^\prime(x)=e^{\int_0^t b^\prime (X_s(x))\, ds}.
\end{equation}
Note that Formula \eqref{eq1.6} is not  true for $d\geq 2$,  that is to say,
\begin{equation*}
Z_t(x)\neq  \exp\Bigl(\int_0^t (\partial_x b)(X_s(x))\, ds\Bigr).
\end{equation*}
\vskip 2mm

In this paper, we only consider the case $d=1$; the use of \eqref{eq1.6}  will considerably simplify the exposition.

\vskip 2mm

After using  the transformation of Girsanov, we will be concerned with the term
\begin{equation}\label{eq1.7}
\E\Bigl[\Bigl( \int_0^t b^\prime(W_s)\, ds\Bigr)^n\Bigr]=\int_{[0,t]^n} \E\Bigl(b^\prime(W_{s_1})\cdots b^\prime(W_{s_n}) \Bigr)\, ds_1\cdots ds_n.
\end{equation}

Consider the simplex
\begin{equation*}
\Delta_n=\bigl\{ (s_1, \ldots, s_n)\in [0,t]^n\ |\ 0<s_1<\cdots <s_n<t\bigr\},
\end{equation*}
and denote
\begin{equation}\label{eq1.8}
I_n(t)=\int_0^t\!ds_1\!\int_{s_1}^t ds_2\ldots ds_n  \, \E\Bigl(b^\prime(W_{s_1})\cdots b^\prime(W_{s_n}) \Bigr).
\end{equation}

In section \ref{sect2},  we show  that using Malliavin calculus, there is a random variable, says $\Lambda_{s_1\ldots s_n}$, such that
\begin{equation}\label{eq1.8.1}
 \E\Bigl(b^\prime(W_{s_1})\cdots b^\prime(W_{s_n}) \Bigr)= \E\Bigl(b(W_{s_1})\cdots b(W_{s_n})\,\Lambda_{s_1\ldots s_n} \Bigr).
\end{equation}

Unfortunately, as we will remark in  section \ref{sect3},  the function $(s_1, \ldots, s_n)\ra \E\bigl(|\Lambda_{s_1\ldots s_n}|\bigr)$ is not integrable over $\Delta_n$.
Therefore estimating $I_n$ in terms of $||b||_\infty$ is a great challenge.
The following estimate was established in (\cite{Davie} , Proposition 2.2): for a constant $M>0$,
\begin{equation}\label{eq1.9}
|I_n(t)|\leq \frac{M^n}{[n/2]!}\, ||b||_\infty^n\, t^{n/2},\quad n\in \N.
\end{equation}

Many authors in literature said that the proof of above inequality is extremely complicated and sophisticated as well.
This beautiful and powerful inequality is worth of a good understanding, a very suitable description of the constant $M$ in \eqref{eq1.9}
is interesting, maybe helpful to obtain connections with other fields.
\vskip 2mm
In the whole paper, we denote by $q_t(x)$ the heat kernel, that is,
\begin{equation*}\label{eq2.13}
q_t(x)=\frac{e^{-x^2/(2t)}}{\sqrt{2\pi t}},\quad t>0, x\in \R.
\end{equation*}

\vskip 2mm
The purpose of the paper is to try to do Inequality  \eqref{eq1.9} from different points of view. The organisation of the paper is as follows.
In section \ref{sect2}, we present basic elements in Malliavin calculus in order to obtain Formula \eqref{eq1.8.1}. As a byproduct, we obtain
a representation formula for the mixed partial derivative of the product  $\dis q_{s_1}(y_1)q_{s_2-s_1}(y_2-y_1)\cdots q_{s_n-s_{n-1}}(y_n-y_{n-1})$.
Section \ref{sect3} is devoted to a digestion of the general expression
$\Lambda_{s_1\ldots s_n}$ for $n=2$ and $n=3$, the estimate of $I_2(t)$ or $I_3(t)$ is not difficult using the heat equation for $q_t(x)$.
Nevertheless, when $n=4$,  a term involving twice second order derivatives of the heat kernels  in the product appears.
Surprisingly enough, such kind of products were not discussed in \cite{Davie} nor in \cite{Proske5}.  In section \ref{sect4}, we show that it is
difficult to transfer this type of products into other types of products considered in \cite{Davie, Proske5}; but fortunately, a key lemma
in their papers can be used to finally overcome the difficulty. In section \ref{sect5},  in order to get the best rate of decrease in $n$ of $I_n(t)$,
we consider the term only involving the first order derivative of heat kernels in $I_n(t)$, we get an explicit expression of the upper bound constant
which is related to the volume  $v_n$ of the unit ball in $\R^n$.
In section \ref{sect6}, we obtain an explicit estimate in $L^p$ for $X_t^\prime$.  Since the $L^p$ norm of $X_t^\prime$ (see \eqref{eq5.3} and \eqref{eq6.3})
  gives rise an upper bound which is of order $e^{Cp}$, we do not get here the Osgood continuity;
  instead we prove that SDE \eqref{eq1.1} defines a continuous flow
of maps in Sobolev space $W_1^p$. Finally in section \ref{sect7}, we consider $X_t$ as a functional on
the Wiener space, a result on Malliavin derivability is obtained.


\section{Malliavin calculus and products of heat kernels }\label{sect2}

Malliavin Calculus was introduced by P. Malliavin in \cite{Malliavin1} in order to establish an integration by parts formula for the term
\begin{equation*}
\E\Bigl( (\partial_a\varphi)(F)\Bigr)
\end{equation*}
where $\varphi$ is a $C^1$ function on $\R^n$,  the vector $a$ is fixed  in $\R^n$ and $F$ is a non-degenerate functional. 
We refer the readers to some of the very early papers on Malliavin calculus \cite{Kusuoka, KusuokaS1, KusuokaS2, Nualart}. 
For the purpose of this work,
we consider the Wiener space $\Omega$ of real valued Brownian motion $\{W_t; t\in [0,T]\}$ with a fixed $T>0$, that is,
$\dis \Omega=C([0,T], \R)$ endowed with the Wiener measure $\mu$.  From now on, we use dot to denote the derivative with respect to the time $t$.
The Cameron-Martin subspace $H$  of $\Omega$ is as follows

\begin{equation*}
H=\Bigl\{h\in \Omega;\  \int_0^T |\dot h(s)|^2\, ds<+\infty \Bigr\},
\end{equation*}
 equipped with the inner product
 \begin{equation*}
 \bigl< h_1,h_2\bigr>_H=\int_0^T \dot h_1(s)\dot h_2(s)\, ds.
 \end{equation*}

A functional $F: \Omega\ra \R^n$ is said to belong to the Sobolev space $\D_1^p(\Omega, \R^n)$ if the gradient $\nabla F: \Omega\ra H\otimes\R^n$ of $F$ exists in $L^p$ such that
the limit of
\begin{equation*}
\frac{\tau_{\eps h}F-F}{\eps}- \<\nabla F,h\>_H
\end{equation*}
holds in $L^{p-}$ as $\eps\ra 0$, where $\tau_{\eps h}$ denotes the translation by $\eps h$.
We denote also by $D_hF$ the directional derivative of $F$ along $h$, that is,  $D_hF=\<\nabla F, h\>_H$.
Let $G$ be a separable Hilbert space. A $G$-valued functional $Z$ is said in the Sobolev space $\D_1^p(\Omega, G)$ if there exists
$\nabla Z\in L^p(\Omega, H\otimes G)$ such that

\begin{equation*}
\frac{\tau_{\eps h}Z-Z}{\eps}- \<\nabla Z,h\>_H
\end{equation*}
holds in $\dis L^{p-}(\Omega, G)$.  A vector field $Z$ on $\Omega$ is a $H$-valued functional  in some Sobolev space $\D_1^p$.
\vskip 2mm

For a functional $F\in \D_1^p(\Omega, \R^n)$ in the form  $F=(F^1, \ldots, F^n)$, its Malliavin covariant matrix is a $n\times n$ matrix defined by
\begin{equation}\label{eq2.1}
\sigma_F(w) = \Bigl( \<\nabla F^i(w), \nabla F^j(w)\>_H\Bigr)_{1\leq i,j\leq n}.
\end{equation}
 The functional $F$ is said to be non-degenerate if $\dis \sigma_F^{-1}$ exists almost surely  and $\textup{det}\,\sigma_F^{-1}\in L^p(\Omega)$. The
 gradient $\nabla F(w)$ can be seen as an operator $H\ra \R^n$, its adjoint operator $\bigl(\nabla F(w)\bigr)^\ast$ is an operator $\R^n\ra H$ such that
 \begin{equation*}
 \<\bigl(\nabla F(w)\bigr)^\ast\eta, h\>_H=\< D_h F(w), \eta\>_{\R^n},\quad \textup{for any}\ h\in H, \eta\in \R^n.
 \end{equation*}

 For a fixed vector $\eta\in \R^n$, we introduce $h_\eta\in H$ defined by
 \begin{equation}\label{eq2.2}
 h_\eta=\bigl(\nabla F(w)\bigr)^\ast \sigma_F^{-1}(w)\eta.
 \end{equation}

Then we have (see for example \cite{Fang1}),
\begin{equation*}\label{eq2.3}
\bigl( D_{h_\eta}F\bigr)(w)=\eta,
\end{equation*}
and accordingly
\begin{equation}\label{eq2.4}
D_{h_\eta}\bigl( \varphi(F)\bigr)=\varphi^\prime(F)\,D_{h_\eta}F=\varphi^\prime (F)\eta=(\partial_\eta\varphi)(F),
\end{equation}
where $\varphi^\prime$ denotes the differential of $\varphi$.
If $h_\eta$ is in the Sobolev space $\D_1^p(\Omega, H)$ for some $p>1$, then the so-called  divergence of $h_\eta$, the term $\delta(h_\eta)$
exists in $L^P(\Omega)$.
The formula of integration by parts in Malliavin calculus (see \cite{Malliavin2, Watanabe,  Fang1, Nualart}) yields
\begin{equation}\label{eq2.5}
\E\Bigl( \bigl(\partial_\eta \varphi\bigr)(F)\Bigr)
=\E\Bigl( \varphi(F)\delta(h_\eta)\Bigr).
\end{equation}

\vskip 2mm
Let $n\in \N$. By Fubini theorem, we have
\begin{equation*}
\begin{split}
\E\Bigl[ \Bigl(\int_0^t b^\prime(W_s)\, ds\Bigr)^n\Bigr]
&=\E\Bigl( \int_{[0,t]^n}  b^\prime(W_{s_1})\cdots  b^\prime(W_{s_n})\, ds_1\cdots ds_n\Bigr)\\
&= \int_{[0,t]^n}   \E\Bigl(b^\prime(W_{s_1})\cdots  b^\prime(W_{s_n})\Bigr)\, ds_1\cdots ds_n.
\end{split}
\end{equation*}

Consider the functional $F: \Omega\ra \R^n$ defined by
\begin{equation}\label{eq2.6}
F=(W_{s_1}, \cdots, W_{s_n}).
\end{equation}

\begin{proposition}\label{prop2.1}
i)\quad The Malliavin covariance matrix of $F$ is given by
\begin{equation}\label{eq2.7}
\sigma_F=(s_i\wedge s_j)_{1\leq i, j\leq n}.
\end{equation}

ii)\quad Denote $\dis \sigma_F^{-1}=(\sigma_{ij}^{-1})_{1\leq i,j\leq n}$ and $\{e_1, \ldots, e_n\}$ the canonical basis of $\R^n$, then the element $h_j=h_{e_j}$
for $j=1, \ldots, n$ defined in \eqref{eq2.2} admits the expression
\begin{equation}\label{eq2.8}
h_j(\tau)=\sum_{i=1}^n \sigma_{ij}^{-1}\, (s_i\wedge\tau), \quad \tau \in [0,T].
\end{equation}
\end{proposition}

\begin{proof} Let $h\in H$; we compute the derivative of $F$ along $h$, namely $D_hF$, in the following
\begin{equation*}
(D_hF)(w)=\frac{d}{d\eps}_{|_{\eps=0}} F(w+\eps h)=(h_{s_1}, \ldots, h_{s_n}).
\end{equation*}
Let $\zeta_{s_i}\in H$ such that $\dis \dot \zeta_{s_i}(\tau)={\bf 1}_{(\tau<s_i)}$ or
$\dis \zeta_{s_i}(\tau)=\tau\wedge s_i$ for $\tau\in [0,T]$. Then for each $h\in H$, $\dis (D_hF^i)(w)=\<\zeta_{s_i},h\>_H$ so that
\begin{equation*}
(\nabla F^i)(w)=\zeta_{s_i},
\end{equation*}
and
\begin{equation*}
\<\nabla F^i(w), \nabla F^j(w)\>_H=\int_0^T \dot \zeta_{s_i}(\tau)\dot\zeta_{s_j}(\tau)\, d\tau=s_i\wedge s_j.
\end{equation*}
Relation \eqref{eq2.7} follows. Now for $\eta\in \R^n$, we have
\begin{equation*}
\bigl(\nabla F(w)\bigr)^\ast\eta=\sum_{j=1}^n \eta_j \nabla F^j(w)=\sum_{j=1}^n \eta_j \zeta_{s_j}.
\end{equation*}
Then, according to expression \eqref{eq2.2},
\begin{equation*}
h_j=\bigl(\nabla F(w)\bigr)^\ast\sigma_F^{-1}e_j=\sum_{i=1}^n (\sigma_F^{-1}e_j)^i \zeta_{s_i}=\sum_{i=1}^n \sigma_{ij}^{-1}\zeta_{s_i}.
\end{equation*}
We get Expression \eqref{eq2.8}.
\end{proof}

\vskip 2mm
We define now the function $\Phi$ on $\R^n$ by $\Phi(y_1, \ldots, y_n)=b(y_1)\ldots b(y_n)$. Then
\begin{equation*}\label{eq2.9}
\frac{\partial^n\Phi}{\partial{y_1}\ldots\partial{y_n}}(y_1, \ldots, y_n)=b^\prime(y_1)\cdots b^\prime(y_n).
\end{equation*}

Using Relation \eqref{eq2.4}, we have the expression
\begin{equation*}
\frac{\partial^n\Phi}{\partial{y_1}\ldots\partial{y_n}}(F)=D_{h_1}\Bigl( \frac{\partial^{n-1}\Phi}{\partial{y_2}\ldots\partial{y_n}}(F)\Bigr).
\end{equation*}
Therefore applying integration by parts formula \eqref{eq2.5} on the Wiener space $\Omega$, we obtain
\begin{equation*}
\E\Bigl(\frac{\partial^n\Phi}{\partial{y_1}\ldots\partial{y_n}}(F)\Bigr)
=\E\Bigl(\frac{\partial^{n-1}\Phi}{\partial{y_2}\ldots\partial{y_n}}(F)\,\delta(h_1)\Bigr).
\end{equation*}

Using again \eqref{eq2.4},
\begin{equation*}
\begin{split}
\frac{\partial^{n-1}\Phi}{\partial{y_2}\ldots\partial{y_n}}(F)\,\delta(h_1)
&= D_{h_2}\Bigl[ \frac{\partial^{n-2}\Phi}{\partial{y_3}\ldots\partial{y_n}}(F)\Bigr]\,\delta(h_1)\\
&=\< \nabla \Bigl[ \frac{\partial^{n-2}\Phi}{\partial{y_3}\ldots\partial{y_n}}(F)\Bigr],\ \delta(h_1)h_2\>_H,
\end{split}
\end{equation*}
which yields, via again \eqref{eq2.5},
\begin{equation*}
\E\Bigl( \frac{\partial^{n-2}\Phi}{\partial{y_3}\ldots\partial{y_n}}(F)\, \delta\bigl(\delta(h_1)h_2\bigr)\Bigr).
\end{equation*}

Repeatedly in this way, we finally get the following relation
\begin{equation}\label{eq2.10}
\E\Bigl(\frac{\partial^n\Phi}{\partial{y_1}\ldots\partial{y_n}}(F)\Bigr)
=\E\Bigl( \Phi(F)\, \delta\bigl(h_n\delta(h_{n-1}\cdots \delta(h_2\delta(h_1))\cdots\bigr)\Bigr).
\end{equation}

\begin{proposition}\label{prop2.2}
Let
\begin{equation}\label{eq2.11}
 \Lambda_{s_1\ldots s_n}=\delta\bigl(h_n\delta(h_{n-1}\cdots \delta(h_2\delta(h_1))\cdots\bigr).
\end{equation}
Then
\begin{equation}\label{eq2.12}
\E\Bigl[ \Bigl(\int_0^t b^\prime(W_s)\,ds\Bigr)^n\Bigr]=\int_{[0,t]^n}\E\Bigl(b(W_{s_1})\cdots b(W_{s_n})\,  \Lambda_{s_1\ldots s_n}\Bigr)\, ds_1\cdots ds_n.
\end{equation}
\end{proposition}

\begin{proof} Relation \eqref{eq2.12} follows from Fubini theorem and Equality \eqref{eq2.10}.
\end{proof}

\vskip 2mm
As a byproduct of above considerations, we get the following representation formula

\begin{theorem}\label{th2.1}  For $0<s_1<\cdots < s_n<t$ and $(y_1, \ldots, y_n)\in \R^n$, denote
\begin{equation}\label{eq2.13}
 Q_{s_1\ldots s_n}(y_1, \ldots, y_n)=q_{s_1}(y_1)q_{s_2-s_1}(y_2-y_1)\cdots q_{s_n-s_{n-1}}(y_n-y_{n-1}).
 \end{equation}

 Then we have
\begin{equation}\label{eq2.14}
\Bigl[\bigl(Q_{s_1\ldots s_n}\bigr)^{-1}\, \frac{\partial^n Q_{s_1\ldots s_n}}{\partial{y_n}\cdot\cdot\partial{y_1}}\Bigr](y_1, \ldots, y_n)
=(-1)^n\E\Bigl(\Lambda_{s_1\ldots s_n}|W_{s_1}=y_1, \ldots,  W_{s_n}=y_n\Bigr),
\end{equation}
where $\Lambda_{s_1\ldots s_n}$ is defined in \eqref{eq2.11}.
\end{theorem}

\begin{proof} Let $v_1, \ldots, v_n\in C_c^1(\R)$ be functions on $\R$ with compact support; the totality
of tensor products $v_1\otimes \cdots \otimes v_n$  of such functions
is dense in  $L^p(\R^n)$.  Define $\dis \Psi(y_1, \ldots,  y_n)=v_1(y_1)\cdots v_n(y_n)$, we have
\begin{equation*}
\frac{\partial^n\Psi}{\partial{y_1}\ldots\partial{y_n}}(y_1, \ldots, y_n)=v_1^\prime(y_1)\cdots v_n^\prime(y_n).
\end{equation*}
Then proceeding as above, we have on one hand,
\begin{equation*}
\begin{split}
&\E\Bigl( \frac{\partial^n\Psi}{\partial{y_1}\ldots\partial{y_n}}(W_{s_1}, \cdots, W_{s_n})\Bigr)
= \E\Bigl( \Psi(W_{s_1}, \cdots, W_{s_n})\,\Lambda_{s_1\ldots s_n}\Bigr)\\
&=\int_{\R^n} \Psi(y_1, \ldots,  y_n)\E\bigl(\Lambda_{s_1\ldots s_n}|W_{s_1}=y_1,\cdots, W_{s_n}=y_n\bigr)  Q_{s_1\ldots s_n}(y_1, \ldots, y_n)\,dy_1\cdots dy_n,
\end{split}
\end{equation*}
since the law of $\dis (W_{s_1}, \cdots, W_{s_n})$ has $\dis  Q_{s_1\ldots s_n}$ as density.  On the other hand,

\begin{equation*}
\E\Bigl( \frac{\partial^n\Psi}{\partial{y_1}\ldots\partial{y_n}}(W_{s_1}, \cdots, W_{s_n})\Bigr)
=\int_{\R^n}  \frac{\partial^n\Psi}{\partial{y_1}\ldots\partial{y_n}}\, Q_{s_1\ldots s_n}\, dy_1\cdots dy_n,
\end{equation*}
which leads, via $n$ times integration by parts, to
\begin{equation*}
(-1)^n \int_{\R^n} \Psi\,  \frac{\partial^n Q_{s_1\ldots s_n}}{\partial{y_1}\ldots\partial{y_n}}\, dy_1\cdots dy_n.
\end{equation*}
The result \eqref{eq2.14} follows.
\end{proof}

\section{Digestion for $n=2$ and $n=3$}\label{sect3}

In this part, we do some explicit computations for cases $n=2$ and $n=3$ in order to get feelings about the term $\dis\Lambda_{s_1\ldots s_n}$.
  We recall formulae in Malliavin calculus alongside where they are needed.
\vskip 2mm

For $n=2$, consider $0<s_1<s_2<T$. In this case, the Malliavin covariance matrix $\sigma_F$ of $F=(W_{s_1}, W_{s_2})$
is $\dis \sigma_F=\left(\begin{array}{cc}s_1&s_1\\s_1&s_2\end{array}\right)$, the inverse $\sigma_F^{-1}$ of $\sigma_F$ is given by

\begin{equation}\label{eq3.1}
\sigma_F^{-1}=\left( \begin{array}{cc} \sigma_{11}^{-1}&\ \sigma_{12}^{-1}\\
\\
							\sigma_{21}^{-1}&\ \sigma_{22}^{-1}\end{array}\right)
		= \left( \begin{array}{cc} \frac{s_2}{s_1(s_2-s_1)}&\ \frac{-1}{s_2-s_1}\\
		\\
		\frac{-1}{s_2-s_1}&\ \frac{1}{s_2-s_1}\end{array}\right).
	\end{equation}

By \eqref{eq2.8}, we have
\begin{equation}\label{eq3.2}
h_1(\tau)=\sigma_{11}^{-1}\, (s_1\wedge\tau) + \sigma_{21}^{-1}\, (s_2\wedge\tau)\quad \textup{and}\quad
h_2(\tau)=\sigma_{12}^{-1}\, (s_1\wedge\tau) + \sigma_{22}^{-1}\, (s_2\wedge\tau).
\end{equation}

The divergence $\delta(h)$ of $h\in H$ is It\^o stochastic integral of $\dot h$, according to \eqref{eq3.2},  we have

\begin{equation}\label{eq3.3}
\delta(h_1) =\sigma_{11}^{-1}\, W_{s_1}+ \sigma_{21}^{-1}\, W_{s_2}\quad \textup{and} \quad
\delta(h_2)=\sigma_{12}^{-1}\, W_{s_1} + \sigma_{22}^{-1}\, W_{s_2}.
\end{equation}

Let $Z_1, Z_2$ be two vector fields on the Wiener space $\Omega$, the following formula (see \cite{Malliavin2, Fang1})
holds:

\begin{equation}\label{eq3.4}
\delta\bigl( Z_2\,\delta(Z_1)\bigr)=\delta(Z_2)\delta(Z_1)-D_{Z_2}\delta(Z_1).
\end{equation}
For $Z_2=h_2, Z_1=h_1$, we have
\begin{equation}\label{eq3.5}
D_{h_2}\delta(h_1)=\<h_1, h_2\>_H.
\end{equation}

Using expression \eqref{eq3.2}, we have
\begin{equation*}
\<h_1,h_2\>_H=\bigl( \sigma_{11}^{-1}\sigma_{12}^{-1}+\sigma_{11}^{-1}\sigma_{22}^{-1}+\sigma_{12}^{-1}\sigma_{21}^{-1}\bigr)\, s_1
+\sigma_{21}^{-1}\sigma_{22}^{-1}\, s_2,
\end{equation*}
as well as, by \eqref{eq3.3},
\begin{equation*}
\delta(h_1)\delta(h_2)= \sigma_{11}^{-1}\sigma_{12}^{-1}\, W_{s_1}^2+\bigl(\sigma_{11}^{-1}\sigma_{22}^{-1}+\sigma_{12}^{-1}\sigma_{21}^{-1}\bigr)\, W_{s_1}W_{s_2}
+\sigma_{21}^{-1}\sigma_{22}^{-1}\, W_{s_2}.
\end{equation*}

Combining these formulae, we get
\begin{equation}\label{eq3.6}
\begin{split}
\delta\bigl(h_2\delta(h_1)\bigr)=  \sigma_{11}^{-1}\sigma_{12}^{-1}\, (W_{s_1}^2-s_1)
&+\bigl(\sigma_{11}^{-1}\sigma_{22}^{-1}+\sigma_{12}^{-1}\sigma_{21}^{-1}\bigr)\, \bigl(W_{s_1}W_{s_2}-s_1\bigr)\\
&+\sigma_{21}^{-1}\sigma_{22}^{-1}\, (W_{s_2}-s_2).
\end{split}
\end{equation}

Using expressions in \eqref{eq3.1}, we have
\begin{equation*}
 \sigma_{11}^{-1} \sigma_{12}^{-1}=-\frac{s_2}{s_1(s_2-s_1)}, \
 \sigma_{12}^{-1} \sigma_{21}^{-1}+ \sigma_{11}^{-1} \sigma_{22}^{-1}=\frac{1}{(s_2-s_1)^2}+\frac{s_2}{s_1(s_2-s_1)},
\end{equation*}
and
\begin{equation*}
 \sigma_{21}^{-1} \sigma_{22}^{-1}=-\frac{1}{(s_2-s_1)^2}.
\end{equation*}
Therefore for $n=2$, we have

\begin{equation*}\label{eq3.7}
\Lambda_{s_1s_2}=-\frac{s_2(W_{s_1}^2-s_1)}{s_1(s_2-s_1)^2}+\Bigl[ \frac{1}{(s_2-s_1)^2}+\frac{s_2}{s_1(s_2-s_1)}\Bigr]\bigl(W_{s_1}W_{s_2}-s_1\bigr)
- \frac{W_{s_2}^2-s_2}{(s_2-s_1)^2}.
\end{equation*}


In the sequel, taking account of independence of increments of the Brownian motion,
we express $\dis \Lambda_{s_1s_2}$ in the following

\begin{proposition}\label{prop3.1} We have
\begin{equation}\label{eq3.8}
\Lambda_{s_1s_2}=\frac{W_{s_1}(W_{s_2}-W_{s_1})}{s_1(s_2-s_1)}-\Bigl(\frac{(W_{s_2}-W_{s_1})^2}{(s_2-s_1)^2}-\frac{1}{s_2-s_1}\Bigr).
\end{equation}
\end{proposition}

\begin{proof} It is sufficient to identify coefficients of $W_{s_1}^2, (W_{s_2}-W_{s_1})W_{s_1}$ and $(W_{s_2}-W_{s_1})^2$ from \eqref{eq3.6}.
The coefficient of $W_{s_1}^2$ is
\begin{equation*}
 \sigma_{11}^{-1}\sigma_{12}^{-1}+\sigma_{12}^{-1}\sigma_{21}^{-1}+\sigma_{11}^{-1}\sigma_{22}^{-1}
 +\sigma_{21}^{-1}\sigma_{22}^{-1}=0.
\end{equation*}
The coefficient of $W_{s_1}(W_{s_2}-W_{s_1)}$ is
\begin{equation*}
\sigma_{12}^{-1}\sigma_{21}^{-1}+\sigma_{11}^{-1}\sigma_{22}^{-1}+2\sigma_{21}^{-1}\sigma_{22}^{-1}
=\frac{1}{s_1(s_2-s_1)},
\end{equation*}
while the coefficient of $(W_{s_2}-W_{s_1})^2$ is
\begin{equation*}
\sigma_{21}^{-1}\sigma_{22}^{-1}=-\frac{1}{(s_2-s_1)^2}.
\end{equation*}
The coefficient of $s_1$ is the same as the coefficient of $W_{s_1}^2$, which is equal to $0$; in the same way, the coefficient of $s_2-s_1$
is
\begin{equation*}
-\sigma_{21}^{-1}\sigma_{22}^{-1}=\frac{1}{(s_2-s_1)^2}.
\end{equation*}
Therefore we obtain \eqref{eq3.8} from \eqref{eq3.6}.
\end{proof}

\vskip 2mm

For $n=3$, consider $0<s_1<s_2<s_3<T$ and $\dis F=(W_{s_1}, W_{s_2}, W_{s_3})$. The Malliavin covariance matrix $\sigma_F$ of $F$ is given by
\begin{equation*}
\sigma_F=\left(\begin{array}{ccc} s_1\ &s_1\ &s_1\\
s_1\ &s_2\ &s_2
\\s_1\ &s_2\ &s_3\end{array}\right).
\end{equation*}
The inverse $\sigma_F^{-1}$ of $\sigma_F$ admits the expression
\begin{equation*}\label{eq3.9}
\sigma_F^{-1}=\left(\begin{array}{ccc} \sigma_{11}^{-1}&\ \sigma_{12}^{-1}&\ \sigma_{13}^{-1}\\
\\
\sigma_{21}^{-1}&\ \sigma_{22}^{-1}&\ \sigma_{23}^{-1}\\
\\
\sigma_{31}^{-1}&\ \sigma_{32}^{-1}&\ \sigma_{33}^{-1}\end{array}\right)
=\left(\begin{array}{ccc}\frac{s_2}{s_1(s_2-s_1)}&\ \frac{-1}{s_2-s_1}&\ 0\\
\\
\frac{-1}{s_2-s_1}&\ \frac{s_3-s_1}{(s_2-s_1)(s_3-s_2)}&\ \frac{-1}{s_3-s_2}\\
\\
0&\ \frac{-1}{s_3-s_2}&\ \frac{1}{s_3-s_2}\end{array}\right).
\end{equation*}

By \eqref{eq2.8}, we have for $j=1, 2, 3$,
\begin{equation}\label{eq3.10}
h_j(\tau)=\sum_{i=1}^3 \sigma_{ij}^{-1}\, (s_i\wedge\tau),
\end{equation}
and
\begin{equation}\label{eq3.11}
\delta (h_j) =\sum_{i=1}^3 \sigma_{ij}^{-1}\, W_{s_i}.
\end{equation}

Using \eqref{eq3.4} and \eqref{eq3.5}, we have
\begin{equation*}
\delta\bigl(h_2\delta(h_1)\bigr)=\delta(h_2)\delta(h_1)-\<h_1,h_2\>_H.
\end{equation*}

Again by \eqref{eq3.4} and \eqref{eq3.5}, we have
\begin{equation*}
\Lambda_{s_1s_2s_3}=\delta(h_1)\delta(h_2)\delta(h_3)-\delta(h_3)\<h_1,h_2\>_H-D_{h_3}\bigl( \delta(h_1)\delta(h_2)\bigr).
\end{equation*}
Note that
\begin{equation*}
D_{h_3}\bigl( \delta(h_1)\delta(h_2)\bigr)=\<h_1,h_3\>_H\delta(h_2)+\<h_2,h_3\>_H\delta(h_1).
\end{equation*}

Finally using \eqref{eq3.10} and \eqref{eq3.11}, we get
\begin{equation*}
\begin{split}
\<h_1,h_2\>_H\delta(h_3)&=\< \sum_{i=1}^3\sigma_{i1}^{-1}(s_i\wedge \tau), \sum_{i=1}^3\sigma_{i2}^{-1}(s_i\wedge \tau)\>_H\,
\bigl(\sum_{i=1}^3\sigma_{i3}^{-1}\  W_{s_i}\bigr)\\
&=\sum_{i_1,i_2, i_3=1}^3 \bigl( \sigma_{i_11}^{-1}\sigma_{i_22}^{-1}\sigma_{i_33}^{-1}\bigr)\, (s_{i_1}\wedge s_{i_2})\, W_{s_{i_3}}.
\end{split}
\end{equation*}
This term does not change under permutation $1\ra 2\ra 3\ra 1$.
\vskip 2mm

\begin{proposition}\label{prop3.2} The term $\Lambda_{s_1s_2s_3}$ admits the following expression
\begin{equation}\label{eq3.12}
\Lambda_{s_1s_2s_3}=\sum_{i_1,i_2, i_3=1}^3\bigl( \sigma_{i_11}^{-1}\sigma_{i_22}^{-1}\sigma_{i_33}^{-1}\bigr)
\Bigl(W_{s_{i_1}}W_{s_{i_2}}W_{s_{i_3}}-3\bigl(s_{i_1}\wedge s_{i_2}\bigr) W_{s_{i_3}}\Bigr).
\end{equation}
\end{proposition}

\begin{proof} Developing the product $\delta(h_1)\delta(h_2)\delta(h_3)$ gives the first sum in \eqref{eq3.12}.
\end{proof}

\vskip 2mm
We see here that $\Lambda_{s_1s_2s_3}$ is a polynomial  of $W_{s_1}, W_{s_2},  W_{s_3}$ of degree $3$.
Having Formula \eqref{eq2.12} in hand, one asks whether the term $\dis \E(|\Lambda_{s_1\ldots s_n}|)$ is integrable over the simplex
\begin{equation}\label{eq3.13}
\Delta_n=\bigl\{ (s_1, \ldots, s_n)\in [0,t]^n|\ 0<s_1<\cdots <s_n<t\bigr\}.
\end{equation}

Unfortunately, it is not true. We illustrate this point in the following for $n=2$.

\begin{proposition} We have, for $0<s_1<s_2<t$,
\begin{equation*}
\int_0^t ds_1\int_{s_1}^t \E(|\Lambda{s_1s_2}|)\, ds_2=+\infty.
\end{equation*}
\end{proposition}

\begin{proof} Let $\dis G=\frac{W_{s_2}-W_{s_1}}{\sqrt{s_2-s_1}}$, that is a standard Gaussian random variable. It is obvious that
the first term in \eqref{eq3.8} is integrable, while the second term is equal to
\begin{equation*}
\frac{1}{s_2-s_1} \bigl( G^2-1\bigr).
\end{equation*}
We see that
\begin{equation*}
\int_0^t ds_1\int_{s_1}^t \frac{ds_2}{s_2-s_1}\E(|G^2-1|) =+\infty,
\end{equation*}
which gives the desired result.
\end{proof}

Therefore the estimate of the right hand side of \eqref{eq2.12} is delicate over the simplex $\Delta_n$, as what we show
in the sequel.
\vskip 2mm

For $n\in \N^\ast$ and $0<s_1<\cdots <s_n<t$, by \eqref{eq2.12}, we have 
\begin{equation}\label{eq3.14}
I_n(t)=\int_0^tds_1\int_{s_1}^tds_2\cdots ds_{n}\int_{s_n}^t \E\Bigl( b(W_{s_1})\cdots b(W_{s_n})\Lambda_{s_1\ldots s_n}\Bigr).
\end{equation}

According to \eqref{eq3.8}, the delicate term in $I_2(t)$ is
\begin{equation*}
I_{2,2}(t)=\int_0^t \!ds_1\!\int_{s_1}^t ds_2\!\!\int_{\R^2} b(y_1)b(y_1+y_2)\Bigl( \frac{y_2^2}{(s_2-s_1)^2}-\frac{1}{s_2-s_1}\Bigr)\, q_{s_1}(y_1)q_{s_2-s_1}(y_2)\, dy_1dy_2.
\end{equation*}

Note that
\begin{equation*}
2\dot q_t(x)=\bigl(\frac{x^2}{t^2}-\frac{1}{t}\bigr) q_t(x),
\end{equation*}
then
\begin{equation*}
\Bigl( \frac{y_2^2}{(s_2-s_1)^2}-\frac{1}{s_2-s_1}\Bigr)\, q_{s_2-s_1}(y_2)=2\dot q_{s_2-s_1}(y_2).
\end{equation*}
Let $\eps>0$, consider the term

\begin{equation*}
J_\eps = \int_\eps^{t-s_1}\int_\R b(y_1+y_2)\dot q_r(y_2) drdy_2.
\end{equation*}
Then
\begin{equation*}
J_\eps=\int_\R b(y_1+y_2)q_{t-s_1}(y_2)dy_2-\int_\R b(y_1+y_2)q_{\eps}(y_2)dy_2,
\end{equation*}
which converges to, as $\eps\ra 0$,
\begin{equation*}
\int_\R b(y_1+y_2)q_{t-s_1}(y_2)dy_2-b(y_1).
\end{equation*}
It follows that
\begin{equation*}
I_{2,2}(t)= \int_0^t\!\!\int_{\R^2} b(y_1)b(y_1+y_2) q_{s_1}(y_1)q_{t-s_1}(y_2) dy_1dy_2 ds_1
- \int_0^t\!\!\int_{\R} b^2(y_1) q_{s_1}(y_1)dy_1 ds_1,
\end{equation*}
the last term being reduced to the simple integral.  By above expressions, we have
\begin{equation*}
|I_2(t)|\leq 8\, ||b||_{\infty}^2\, t.
\end{equation*}

\section{Estimates for $I_4(t)$}\label{sect4}

When we compute the term
\begin{equation*}
\frac{\partial^n}{\partial y_1\ldots\partial y_n}\Bigl( q_{s_1}(y_1)q_{s_2-s_1}(y_2-y_1)\cdots q_{s_n-s_{n-1}}(y_n-y_{n-1})\Bigr),
\end{equation*}
 as $n$ goes up more than $4$, the situation becomes very complicated. To illustrate this point, for simplicity,
 we shorten the product  $\dis  q_{s_1}(y_1)q_{s_2-s_1}(y_2-y_1)q_{s_3-s_2}(y_3-y_2)q_{s_4-s_3}(y_4-y_3)$ as
  $\dis q_{s_1}q_{s_2-s_1}q_{s_3-s_2}q_{s_4-s_3}$
  and display all terms for $n=4$ in the following
 \begin{equation}\label{eq4.1}
 \begin{split}
& \frac{\partial^4}{\partial y_1\ldots\partial y_4}\Bigl( q_{s_1}q_{s_2-s_1}q_{s_3-s_2}q_{s_4-s_3}\Bigr)\\
 &\hskip 6mm = q_{s_1}^\prime\, q_{s_2-s_1}^\prime\, q_{s_3-s_2}^\prime\, q_{s_4-s_3}^\prime
 -q_{s_1}^\prime\, q_{s_2-s_1}^\prime\, q_{s_3-s_2}\,q_{s_4-s_3}^{\prime\prime}\\
 &\hskip 6mm + q_{s_1}^\prime\, q_{s_2-s_1}\, q_{s_3-s_2}^\prime\, q_{s_4-s_3}^{\prime\prime}
 -q_{s_1}^\prime \,q_{s_2-s_1}\, q_{s_3-s_2}^{\prime\prime} \,q_{s_4-s_3}^\prime\\
 &\hskip 6mm + q_{s_1}\, q_{s_2-s_1}^{\prime\prime}\, q_{s_3-s_2}\,q_{s_4-s_3}^{\prime\prime}
 -q_{s_1} \,q_{s_2-s_1}^{\prime\prime} q_{s_3-s_2}^{\prime} \,q_{s_4-s_3}^{\prime}\\
 &\hskip 6mm + q_{s_1} q_{s_2-s_1}^\prime\, q_{s_3-s_2}^{\prime\prime}\, q_{s_4-s_3}^{\prime}
 -q_{s_1} \,q_{s_2-s_1}^\prime\, q_{s_3-s_2}^{\prime} \,q_{s_4-s_3}^{\prime\prime},
 \end{split}
 \end{equation}
where the argument concerning $q_{s_{i+1}-s_i}$ is $y_{i+1}-y_i$ accordingly.
  In the sequel, for the sake  of comparison, among terms in \eqref{eq4.1}, we specifically deal with following two terms,
\begin{equation}\label{eq4.2}
J_6=\int_0^t\!ds_1\!\!\int_{s_1}^t\!ds_2\!\!\int_{s_2}^t\!ds_3\!\!\int_{s_3}^t\!ds_4\Bigl[\int_{\R^4}b(y_1)b(y_2)b(y_3)b(y_4)\,
q_{s_1} \,q_{s_2-s_1}^{\prime\prime}\, q_{s_3-s_2}^{\prime} \,q_{s_4-s_3}^{\prime}\Bigr]\, dy_1\cdots dy_4,
\end{equation}
and
\begin{equation}\label{eq4.3}
J_5=\int_0^t\!ds_1\!\!\int_{s_1}^t\!ds_2\!\!\int_{s_2}^t\!ds_3\!\!\int_{s_3}^t\!ds_4\Bigl[\int_{\R^4}b(y_1)b(y_2)b(y_3)b(y_4) \,
q_{s_1} \,q_{s_2-s_1}^{\prime\prime}\, q_{s_3-s_2} \,q_{s_4-s_3}^{\prime\prime}\Bigr]\, dy_1\cdots dy_4.
\end{equation}

The  integral  $J_6$ is among types discussed in \cite{Davie, Proske5}, however the term $J_5$ is not considered there since  the product
$q_{s_2-s_1}^{\prime\prime}\, q_{s_3-s_2} \,q_{s_4-s_3}^{\prime\prime}$  is involved. As we see below, their treatments are completely different.
Firstly,

\begin{proposition}\label{eq4.1} We have
\begin{equation}\label{eq4.4}
|J_6|\leq 2^3\, ||b||_\infty^4\Bigl[\int_0^t\! ds_1\!\int_{s_1}^t\frac{ds_2}{\sqrt{s_2-s_1}}\!\int_{s_2}^t\frac{ds_3}{\sqrt{s_3-s_2}}
+\int_0^t\!ds_1\!\int_{s_1}^t \!ds_2\!\int_{s_2}^t\frac{ds_3}{\sqrt{(s_3-s_2)(t-s_3)}}\Bigr].
\end{equation}
\end{proposition}
\begin{proof} We introduce some notations,
\begin{equation}\label{eq4.5}
K_1(s_3,y_3)=\int_{s_3}^t ds_4\!\int_\R b(y_4) q_{s_4-s_3}^\prime(y_4-y_3) dy_4,
\end{equation}
\begin{equation}\label{eq4.6}
K_2(s_2,y_2)=\int_{s_2}^t ds_3\!\int_\R b(y_3) q_{s_3-s_2}^\prime(y_3-y_2) K_1(s_3, y_3) dy_3,
\end{equation}
and
\begin{equation*}\label{eq4.7}
K_3(s_1,y_1)=\int_{s_1}^t ds_2\!\int_\R b(y_2) q_{s_2-s_1}^{\prime\prime}(y_2-y_1) K_2(s_2, y_2) dy_2.
\end{equation*}
Then  the term $J_6$ has the following expression
\begin{equation}\label{eq4.8}
J_6=\int_0^tds_1\int_\R b(y_1)q_{s_1}(y_1)K_3(s_1,y_1)\, dy_1.
\end{equation}

Note that  $K_1(t,y_3)=K_2(t,y_2)=K_3(t,y_1)=0$.
Using the heat equation $\dis 2 \dot q_t=q_t^{\prime\prime}$ and remarking for $\eps>0$,
\begin{equation*}
\begin{split}
&\int_\eps^{t-s} \dot q_r(y_2-y_1)K_2(s_1+r),y_2)\, dr\\
&=-q_\eps(y_2-y_1)K_2(s_1+\eps,y_2)-\int_\eps^{t-s_1}q_r(y_2-y_1)\dot K_2(s_1+r,y_2)dr,
\end{split}
\end{equation*}
and letting $\eps\ra 0$, we get

\begin{equation}\label{eq4.9}
K_3(s_1,y_1)=-2\,b(y_1)K_2(s_1,y_1)
-2\int_\R b(y_2)dy_2\!\int_{s_1}^t q_{s_2-s_1}(y_2-y_1)\dot K_2(s_2, y_2)ds_2.
\end{equation}
According to definitions of $K_2$ and of $K_1$, respectively in \eqref{eq4.6} and \eqref{eq4.5},  we have
\begin{equation}\label{eq4.10}
\frac{d}{ds_2}K_2(s_2,y_2)=-\int_{\R^2} b(y_3)b(y_4)dy_3dy_4 \!\int_0^{t-s_2}q_r^\prime(y_3-y_2)q^\prime_{t-s_2-r}(y_4-y_3)\, dr.
\end{equation}
Now using the inequality,
\begin{equation}\label{eq4.11}
|q_t^\prime(x)|\leq \frac{2}{\sqrt{te}}\, q_{2t}(x),
\end{equation}
we get
\begin{equation*}
|K_1(s_3,y_3)|\leq \frac{2}{\sqrt{e}}\, ||b||_\infty\int_0^{t-s_3}\frac{dr}{\sqrt{r}},
\end{equation*}
and
we have
\begin{equation*}
|K_2(s_2,y_2)|\leq (\frac{2}{\sqrt{e} }\, ||b||_\infty)^2 \int_{s_2}^t\frac{ds_3}{\sqrt{s_3-s_2}}\int_{s_3}^t \frac{ds_4}{\sqrt{s_4-s_3}}.
\end{equation*}

Combining this inequality with the first term of $K_3$ in \eqref{eq4.9} and Expression \eqref{eq4.8} of $J_6$,
 we get the first term of \eqref{eq4.4}. By \eqref{eq4.10}, we have
 \begin{equation*}
 |\dot K_2(s_2,y_2)|\leq (\frac{2}{\sqrt{e}}\, ||b||_\infty)^2\, \int_{s_2}^t \frac{ds_3}{\sqrt{(s_3-s_2)(t-s_3)}}.
 \end{equation*}
 Therefore the second term of $K_3$ in \eqref{eq4.9} is dominated by
 \begin{equation*}
 (\frac{2}{\sqrt{e}}||b||_\infty)^3\, \int_{s_1}^tds_2 \int_{s_2}^t \frac{ds_3}{\sqrt{(s_3-s_2)(t-s_3)}},
 \end{equation*}
and the second term in \eqref{eq4.4} follows according to \eqref{eq4.8}.
\end{proof}

The estimate for $J_6$ is quite pleasant, however the estimate for $J_5$ defined in \eqref{eq4.3} is of different nature, much harder
than for $J_6$.

\begin{proposition}\label{prop4.2} There is a constant $C_4>0$ such that  for all $t>0$,
\begin{equation}\label{eq4.12}
|J_5|\leq C_4\, ||b||_\infty^4\, t^2.
\end{equation}
\end{proposition}

\begin{proof}  We proceed in a parallel manner as for estimating $J_6$. Put
\begin{equation}\label{eq4.13}
U_1(s_3,y_3)=\int_{s_3}^t ds_4\!\int_\R b(y_4) q_{s_4-s_3}^{\prime\prime}(y_4-y_3) dy_4,
\end{equation}
\begin{equation*}
U_2(s_2,y_2)=\int_{s_2}^t ds_3\!\int_\R b(y_3) q_{s_3-s_2}(y_3-y_2) U_1(s_3, y_3) dy_3,
\end{equation*}
and
\begin{equation*}
U_3(s_1,y_1)=\int_{s_1}^t ds_2\!\int_\R b(y_2) q_{s_2-s_1}^{\prime\prime}(y_2-y_1) U_2(s_2, y_2) dy_2.
\end{equation*}
Then  the term $J_5$ has the expression
\begin{equation}\label{eq4.14}
J_5=\int_0^tds_1\int_\R b(y_1)q_{s_1}(y_1)U_3(s_1,y_1)\, dy_1.
\end{equation}
Similarly to \eqref{eq4.9}, we also have

\begin{equation}\label{eq4.15}
U_3(s_1,y_1)=-2\,b(y_1)U_2(s_1,y_1)
-2\int_\R b(y_2)dy_2\!\int_{s_1}^t q_{s_2-s_1}(y_2-y_1)\dot U_2(s_2, y_2)ds_2.
\end{equation}
We compute the derivative of $U_2(s_2, y_2)$ with respect to $s_2$; since $U_1(t, y_3)=0$, we have
\begin{equation*}
\dot U_2(s_2, y_2)=-\int_{s_2}^tds_3\int_\R b(y_3)\dot q_{s_3-s_2}(y_3-y_2)U_1(s_3, y_3) dy_3.
\end{equation*}

Using $\dis q_{s_4-s_3}^{\prime\prime}=2\dot q_{s_4-s_3}$ and integrating with respect to the time in \eqref{eq4.13}, we get
\begin{equation*}\label{eq4.16}
U_1(s_3, y_3)=2\int_\R b(y_4)q_{t-s_3}(y_4-y_3)dy_4-2b(y_3).
\end{equation*}

Replacing this expression of $U_1$ in above $\dot U_2(s_2, y_2)$, we get
\begin{equation}\label{eq4.17}
\begin{split}
2 \dot U_2(s_2, y_2)=&-\int_{s_2}^tds_3\int_{\R^2}b(y_3)b(y_4)q_{s_3-s_2}^{\prime\prime}(y_3-y_2)q_{t-s_3}(y_4-y_3)dy_3dy_4\\
&-\int_{s_2}^tds_3\int_\R b^2(y_3) q_{s_3-s_2}^{\prime\prime}(y_3-y_2)dy_3.
\end{split}
\end{equation}
In expression \eqref{eq4.15} of $U_3$, after replacing $\dot U_2$ by expression \eqref{eq4.17}, and putting them together in \eqref{eq4.14}, we get
the following form for $J_5$:

\begin{equation*}
\begin{split}
J_5=&2\int_0^tds_1\int_{\R}b^2(y_1)U_2(s_1, y_1)\, dy_1
-\int_0^td\!s_1\!\!\int_{s_1}^t\!ds_2\!\!\int_{s_2}^td\!s_3\! \\
&\Bigl[\int_{\R^4} b(y_1)b(y_2)b(y_3)b(y_4) q_{s_1}(y_1)q_{s_2-s_1}(y_2-y_1)q_{s_3-s_2}^{\prime\prime}(y_3-y_2)q_{t-s_3}(y_4-y_3)dy_1dy_2dy_3dy_4\Bigr]\\
&-\int_0^td\!s_1\!\!\int_{s_1}^t\!ds_2\!\!\int_{s_2}^td\!s_3\!
\Bigl[\int_{\R^3} b(y_1)b(y_2)b^2(y_3) q_{s_1}(y_1)q_{s_2-s_1}(y_2-y_1)q_{s_3-s_2}^{\prime\prime}(y_3-y_2)dy_1dy_2dy_3\Bigr]\\
&=J_{5,1}+J_{5,2}+J_{5,3} \quad\textup{respectively}.
\end{split}
\end{equation*}

For $J_{5,3}$, we have
\begin{equation*}
\begin{split}
J_{5,3}&=-2\int_0^td\!s_1\!\!\int_{s_1}^t\!ds_2\int_{\R^3}b(y_1)b(y_2)b^2(y_3)q_{s_1}(y_1)q_{s_2-s_1}(y_2-y_1)q_{t-s_2}dy_1dy_2dy_3\\
&+2\int_0^td\!s_1\!\!\int_{s_1}^t\!ds_2\int_{\R^2}b(y_1)b^3(y_2)q_{s_1}(y_1)q_{s_2-s_1}(y_2-y_1)dy_1dy_2,
\end{split}
\end{equation*}
which is easily estimated. However estimating $J_{5,2}$ requires special treatments due to its singular structure. Let
\begin{equation*}
\Phi(s_3,y_3)=b(y_3)\int_\R b(y_4)q_{t-s_3}(y_4-y_3)\, dy_4,
\end{equation*}
and
\begin{equation*}
\Psi(s_2,y_2)=\int_0^{s_2} ds_1\int_\R b(y_1)q_{s_1}(y_1)q_{s_2-s_1}(y_2-y_1)\, dy_1.
\end{equation*}
Then we have $\dis |\Phi(s_3, y_3)|\leq ||b||_\infty^2$ and
\begin{equation*}
\begin{split}
|\Psi(s_2, y_2)|&\leq ||b||_\infty^2\int_0^{s_2}ds_1\int_\R q_{s_1}(y_1)q_{s_2-s_1}(y_2-y_1)dy_1\\
&=||b||_\infty ^2\, s_2\, q_{s_2}(y_2)\leq ||b||_\infty^2 \sqrt{s_2}\, e^{-y_2^2/(2s_2)},
\end{split}
\end{equation*}
since $\dis \int_\R q_{s_1}(y_1)q_{s_2-s_1}(y_2-y_1)dy_1=q_{s_2}(y_2)$. Now we write $J_{5,2}$ in the form
\begin{equation*}\label{eq4.18}
J_{5,2}=-\int_0^tds_2\int_{s_2}^t\Psi(s_2,y_2)\Phi(s_3, y_3)q_{s_3-s_2}^{\prime\prime}(y_3-y_2)ds_3dy_2dy_3.
\end{equation*}

The two conditions in Lemma 3.9, page 774 of \cite{Proske5} being checked above, we can apply it to obtain the following estimate
\begin{equation*}\label{eq4.19}
|J_{5,2}|\leq C\ ||b||_\infty^4\, t^2,
\end{equation*}
for some constant $C>0$.   Combining all of these estimates, we get \eqref{eq4.12}.
\end{proof}

\begin{remark}\label{remark4.1} The lemma  3.9  in \cite{Proske5} plays a key role
in the cited paper, the proof of it is extremely complicated.
 Initially we would like to find an alternative proof of it,  but finally we arrive at estimating $J_5$ by using Lemma 3.9 with helps of Jinlong Wei.
\end{remark}

\section{Estimates for general case}\label{sect5}

Let $n\in \N^\ast$, by \eqref{eq2.14} and \eqref{eq3.14}, we have 
\begin{equation}\label{eq5.1}
I_n(t)=\int_0^t\! ds_1\!\int_{s_1}^t\!ds_2\cdots ds_{n-1}\!\int_{\R^n} b(y_1)\cdot\cdot b(y_n)\, \frac{\partial^n}{\partial y_1\cdot\cdot\partial y_n}
\Bigl(q_{s_1}q_{s_2-s_1}\cdots q_{s_n-s_{n-1}}\Bigr)\, ds_nd y_1\cdots dy_n
\end{equation}
where the argument for $q_{s_i-s_{i-1}}$ is $y_i-y_{i-1}$ for $i=1, \ldots, n$ with $s_0=0$ and $y_0=0$. The rate of decrease with $n$ of the constant $C_n$ in the estmate
\begin{equation*}
|I_n(t)|\leq C_n t^{n/2}
\end{equation*}
has impacts on the integrability of  $X_t^\prime$ defined in \eqref{eq1.6}.
In order to find the best rate of decrease in $n$ of above upper bound $C_n$, consider first the simple case
\begin{equation}\label{eq5.2}
I_{n,1}(t)=\int_0^t\! ds_1\!\int_{s_1}^t\!ds_2\cdots ds_{n}\!\int_{\R^n} b(y_1)\cdot\cdot b(y_n)\,
q_{s_1}^{\prime}q_{s_2-s_1}^{\prime}\cdots  q_{s_n-s_{n-1}}^{\prime}dy_1\cdots dy_n
\end{equation}
and try to getting the optimal upper bound.

\begin{lemma}\label{lem5.2}  Define
\begin{equation}\label{eq5.5}
\eta_n(t)=\int_0^t \frac{ds_1}{\sqrt{s_1}}\int_{s_1}^t\frac{ds_2}{\sqrt{s_2-s_1}}\cdots \int_{s_{n-1}}^t \frac{ds_n}{\sqrt{s_n-s_{n-1}}}.
\end{equation}

Then, for $n\in\N$,
\begin{equation}\label{eq5.5}
\eta_n(t)= v_n\, t^{n/2},
\end{equation}
where $v_n$ is the volume of unit Euclidean ball in $\R^n$.
\end{lemma}

\begin{proof} For any $\alpha\geq 0$, set
\begin{equation}\label{eq5.6}
  A(\alpha)=\int_0^1 \frac{(1-\tau)^\alpha}{\sqrt{\tau}}\, d\tau.
\end{equation}
We first compute
\begin{equation*}
\begin{split}
&\int_{s_{n-2}}^t\frac{ds_{n-1}}{\sqrt{s_{n-1}-s_{n-2}}}\int_{s_{n-1}}^t\frac{ds_n}{\sqrt{s_n-s_{n-1}}}
=2\,\int_{s_{n-2}}^t\frac{\sqrt{t-s_{n-1}}}{\sqrt{s_{n-1}-s_{n-2}}}\, ds_{n-1}\\
&=2\, \int_0^{t-s_{n-2}}\frac{\sqrt{t-s_{n-2}-r}}{\sqrt{r}}\, dr=2(t-s_{n-2})\int_0^1 \frac{\sqrt{1-\tau}}{\sqrt{\tau}}\, d\tau
=2(t-s_{n-2})\, A(\frac{1}{2});
\end{split}
\end{equation*}
then we compute the term
\begin{equation*}
2\, A(\frac{1}{2})\int_{s_{n-3}}^t \frac{t-s_{n-2}}{\sqrt{s_{n-2}-s_{n-3}}}ds_{n-2}=2 A(\frac{1}{2})A(1)\, (t-s_{n-3})^{3/2},
\end{equation*}
and so on, we finally get
\begin{equation}\label{eq5.7}
\eta_n(t) = A(0)A(1)\cdots A(\frac{n-1}{2})\, t^{n/2}.
\end{equation}

Consider the Wallis integral, for $k\in\N$,
\begin{equation*}
W(k)=\int_0^{\pi/2} (\sin\theta)^k\, d\theta.
\end{equation*}

Now using the change of variable $\tau=\sin^2\theta$ in \eqref{eq5.6}, we get the relation
\begin{equation*}
A(\alpha)=2 \int_0^{\pi/2} \bigl(\sin\theta\bigr)^{2\alpha+1}\, d\theta=2\, W(2\alpha+1).
\end{equation*}
Using Wallis integrals, expression \eqref{eq5.7} becomes
\begin{equation*}\label{eq5.8}
\eta_n(t)=  2^n \Bigl(\prod_{k=1}^n W(k)\Bigr)\, t^{n/2}.
\end{equation*}
 It is well known that the volume $v_n$ of unit ball in $\R^n$ admits the relation
\begin{equation*}
v_n=2^n \Bigl(\prod_{k=1}^n W(k)\Bigr).
\end{equation*}
The proof of \eqref{eq5.5} is complete.
\end{proof}

\begin{proposition}\label{prop5.1} We have, for $n\in \N$,
\begin{equation}\label{eq5.3}
|I_{n,1}(t)|\leq \frac{(2\sqrt{\pi})^n}{(\sqrt{e})^n\,[n/2]!}\,||b||_\infty^n\, t^{n/2}.
\end{equation}
\end{proposition}

\begin{proof} First by expression \eqref{eq5.2},   it is easy to see that
\begin{equation*}
|I_{n,1}(t)|\leq ||b||_\infty^n\, \int_0^t\! ds_1\!\int_{s_1}^t\!ds_2\cdots ds_{n}\!\int_{\R^n} \,
|q_{s_1}^{\prime}q_{s_2-s_1}^{\prime}\cdots  q_{s_n-s_{n-1}}^{\prime}| dy_1\cdots dy_n.
\end{equation*}

Using Inequality \eqref{eq4.11},  we get
\begin{equation*}
|I_{n,1}(t)|\leq (\frac{2||b||_\infty}{\sqrt{e}})^n \, \eta_n(t).
\end{equation*}

 Remark that, for $q\in \N$,
\begin{equation*}
v_{2q}=\frac{\pi^q}{q!},\quad v_{2q+1}=\frac{2^{2q+1}\pi^q q!}{(2q+1)!}\leq \frac{\pi^q}{q!}.
\end{equation*}
The result \eqref{eq5.3} follows from \eqref{eq5.5}. $\square$

\end{proof}

\vskip 2mm
We see in previous sections that the passage from $I_2(t)$ to $I_4(t)$ meets a very hard difficulty. Of course,
the same degree of difficulty remains for the passage from $I_n(t)$ to $I_{n+1}(t)$. We state here the result obtained in \cite{Davie}
in the following

\begin{proposition}\label{prop5.3} \cite{Davie}  Let $I_n(t)$ be defined in \eqref{eq5.1}, there is a constant $M>0$ such that
\begin{equation}\label{eq5.9}
|I_n(t)|\leq \frac{M^n}{[n/2]!}\,||b||_\infty^n\, t^{n/2}, \quad t>0.
\end{equation}
\end{proposition}

\section{Solutions of SDE in Sobolev spaces}\label{sect6}

During last two decades, a huge number of papers have been published on the topic of SDE with irregular coefficients, see for instance \cite{BassC, KrylovR, ZhangX2,ZhangX3};
 for some very new developments, see for example \cite {BeckFGM, WeiLW, Krylov1, Krylov2, RocknerZ}.
\vskip 2mm

Consider first the case where $b$ is in $C_b^1(\R)$ and $(X_t(x))_{t\geq 0}$ the solution to SDE \eqref{eq1.1}. Let
\begin{equation*}
N_t = \exp\Bigl( -\int_0^t b(X_s(x))dW_s - \frac{1}{2}\int_0^t b^2(X_s(x))ds\Bigr).
\end{equation*}

Then by Girsanov theorem, under the probability measure $N_T\P$, $(X_t(x))_{t\geq 0}$ is a standard Brownian motion.
For $p\geq 1$, we have, for $t\in [0,T]$,
\begin{equation*}
\begin{split}
&\E\Bigl( e^{p\int_0^t b^\prime(X_s(x))ds} \Bigr)
=\E\Bigl( e^{p\int_0^t b^\prime(X_s(x))ds}\,  N_T^{1/2} N_T^{-1/2}\Bigr)\\
&\leq \Bigl[ \E\Bigl( e^{2p\int_0^t b^\prime(X_s(x))ds}\,  N_T \Bigr)\Bigr]^{1/2}\, \sqrt{\E(N_T^{-1})}.
\end{split}
\end{equation*}
The term $\dis  \sqrt{\E(N_T^{-1})}$ is dominated by $\dis e^{\frac{1}{2}||b||_\infty^2 T}$. Therefore

\begin{equation}\label{eq6.1}
\E\Bigl( e^{p\int_0^t b^\prime(X_s(x))ds} \Bigr)\leq e^{\frac{1}{2}||b||_\infty^2 T}\, \Bigl[\E \Bigl( e^{2p\int_0^t b^\prime(W_s)ds} \Bigr)\Bigr]^{1/2}.
\end{equation}

\begin{lemma}\label{lemma6.1} For any positive real number $x$,
\begin{equation}\label{eq6.2}
\sum_{n=0}^{+\infty} \frac{x^n}{[n/2]!}\leq (1+x)\, e^{x^2}.
\end{equation}
\end{lemma}

\begin{proof} For even integer $n=2q$, $\dis  \frac{x^n}{[n/2]!}=\frac{(x^2)^q}{q!}$; for odd integer $n=2q+1$,
$\dis \frac{x^n}{[n/2]!} \leq x\frac{(x^2)^q}{q!}$. Therefore summing above terms yields \eqref{eq6.2}.
\end{proof}

\begin{theorem}\label{th6.2} Let $p\geq 1$ and for $t\in [0,T]$, the following moment estimate holds
\begin{equation}\label{eq6.3}
\E\bigl( |X_t^\prime(x)|^p\bigr)
\leq (1+2pM\sqrt{t}\,||b||_\infty)\,\exp\Bigl( (1+2p^2M^2)T||b||_\infty^2\Bigr),
\end{equation}
where $M$ is the constant involving in \eqref{eq5.9}.
\end{theorem}

\begin{proof} We have, according to Definition of $I_n(t)$ in \eqref{eq5.1},
\begin{equation*}
\frac{1}{n!} \E\Bigl[ (2p)^n \Bigl(\int_0^t b^\prime(W_s)ds\Bigr)^n\Bigr]
\leq (2p)^n I_n(t),
\end{equation*}
which is dominated, using \eqref{eq5.9}, by
\begin{equation*}
\frac{(2pM\sqrt{t}\, ||b||_\infty)^n}{[n/2]!}.
\end{equation*}

Therefore
\begin{equation*}
\E \Bigl( e^{2p\int_0^t b^\prime(W_s)ds} \Bigr)
\leq \sum_{n=0}^{+\infty} \frac{(2pM\sqrt{t}\, ||b||_\infty)^n}{[n/2]!},
\end{equation*}
the result \eqref{eq6.3} follows by combining \eqref{eq6.1} and \eqref{eq6.2}.
\end{proof}

\vskip 2mm

\begin{proposition}\label{prop6.3}  Let  $b$ be a  bounded function on $\R$ and $X_t$ the strong solution to the SDE
\begin{equation}\label{eq6.0}
dX_t(x)=dW_t + b\bigl(X_t(x)\bigr)\, dt,\quad X_0(x)=x.
\end{equation}
Then for each $t\in [0,T]$ and $R>0$, almost surely $x\ra X_t(x)$ is in the Sobolev space $W_1^p(B_R)$, where $B_R=]-R;R[$
and $p\in ]1,\infty[$.
\end{proposition}

\begin{proof}
Take a smooth non negative function $\varphi\in C_c^\infty(\R)$
with support included in $[-1,1]$ such that $\int_\R \varphi(x)dx=1$. For $n\geq 1$, define $\dis  \varphi_n(x)=n\varphi(nx)$
and $b_n=b\ast\varphi_n$ where $\ast$ is the convolution product; we have $||b_n||_\infty\leq ||b||_\infty$. Consider SDE
\begin{equation*}
d\tilde X_t^n(x)= dW_t + b_n(\tilde X_t^n(x)),\quad \tilde X_0^n(x)=x.
\end{equation*}


By Theorem \ref{th6.2}, there is a constant $C(p, T,||b||_\infty)$ dependent of $T$ and the bound $||b||_\infty$ such that for $t\in [0,T]$ and $n\geq 1$,
\begin{equation}\label{eq6.4}
\sup_{x\in\R}\sup_{n\geq 1}\E\Bigl( |(\tilde X_t^n(x))^\prime|^p\Bigr)\leq C(p,T,||b||_\infty).
\end{equation}

We have
\begin{equation*}
\tilde X_t^n(x)=x+W_t+\int_0^t b_n\bigl(\tilde X_s(x)\bigr)\,ds.
\end{equation*}
Then, for $p>1$, there is a constant $C_p>0$ such that
\begin{equation*}
|\tilde X_t^n(x)|^p \leq C_p\,\Bigl( |x|^p+ |W_t|^p+ ||b||_\infty^p t^p\Bigr).
\end{equation*}
Obviously
\begin{equation*}
\E\Bigl(\sup_{|x|\leq R}|\tilde X_t^n(x)|^p\Bigr) \leq C_p\,\Bigl( |x|^p+ \E\bigl(|W_t|^p\bigr)+ ||b||_\infty^p t^p\Bigr),
\end{equation*}
which implies
\begin{equation}\label{eq6.4.1}
\sup_{n\geq 1}\E\Bigl(\int_{B_R}|\tilde X_t^n(x)|^p\,dx\Bigr) \leq 2R C_p\,\Bigl( |x|^p+ \E\bigl(|W_t|^p\bigr)+ ||b||_\infty^p t^p\Bigr).
\end{equation}

 This means that the sequence $\{\tilde X_t; n\geq 1\}$ is bounded in $\dis L^p(\Omega, L^p(B_R))$.
Following the paper \cite{MeyerP},  $X_t$,  the solution to
SDE \eqref{eq6.0} is a cluster:
more precisely, up to a subsequence, $\tilde X_t^n$ converges to $X_t$ weakly in $\dis L^p(\Omega, L^p(B_R))$.
By Banach-Sacks theorem, up to a subsequence,
\begin{equation*}
\frac{\tilde X_t^1+\cdots + \tilde X_t^n}{n}\ra  X_t
\end{equation*}
strongly in $\dis L^p\bigl(\Omega, L^p(B_R)\bigr)$.
 Let
\begin{equation}\label{eq6.4.2}
 X_t^n=\frac{\tilde X_t^1+\cdots + \tilde X_t^n}{n}.
 \end{equation}

 Then $X_t^n$ converges to $X_t$ in $\dis L^p\bigl(\Omega, L^p(B_R)\bigr)$.
By \eqref{eq6.4.1} and convexity, we get for any $R>0$ and $n\geq 1$,
\begin{equation}\label{eq6.5}
\Bigl[\E\Bigl(\int_{B_R} |(X_t^n)^\prime|^p)\, dx\Bigr)\Bigr]^{1/p}\leq \bigl(2C(p,T,||b||_\infty)\,R\bigr)^{1/p}.
\end{equation}
 Hence the sequence $\{X_t^n;\ n\geq 1\}$ is bounded in
\begin{equation*}
L^p(\Omega, W_1^p(B_R)).
\end{equation*}
 Up to a subsequence, $X_t^n$ converges to $Y_t\in L^p(\Omega, W_1^p(B_R))$
weakly as $n\ra +\infty$. More precisely, for any bounded random variable $\xi$ and $g\in W_1^q(B_R)$, we have
\begin{equation*}
\int_\Omega \xi \, \int_{B_R} \Bigl( X_t^n g+(X_t^n)^\prime g^\prime\Bigr)\, dx\, d\P
\ra \int_\Omega \xi \, \int_{B_R} \Bigl( X_t g+(X_t)^\prime g^\prime\Bigr)\, dx\, d\P.
\end{equation*}

Using Banach Sachs theorem, up to a subsequence again,
\begin{equation*}
\frac{X_t^1+\cdot + X_t^n}{n}
\end{equation*}
converges to $Y_t$ strongly in $\dis  L^p(\Omega, W_1^p(B_R))$. Since the sequence $(X_t^n)$ converges
to $X_t$ in $\dis L^p(\Omega, L^p(B_R))$, by the identification of limits, we get $Y_t=X_t$ in  $\dis L^p(\Omega, L^p(B_R))$;
therefore $X_t\in L^p\bigl(\Omega, W_1^p(B_R)\bigr)$. Since the norm is lower semi-continuous for the weak convergence,
letting $n\ra +\infty$ in \eqref{eq6.5}, we get

\begin{equation*}\label{eq6.6}
\E\Bigl(\int_{B_R} |(X_t)^\prime|^p)\, dx\Bigr)\leq 2C(p,T,||b||_\infty)\,R.
\end{equation*}
It follows  that $X_t\in W_1^p(B_R)$ for each $t\in [0,T]$, almost surely.
\end{proof}

\begin{remark} In an early work \cite{PilipenkoA}, Sobolev regularity for stochastic flows was obtained under some weak derivative conditions
on the drift. Such result was also obtained in \cite{Krylov1} under the condition that the drift is in $L^d$.  We have the following stronger result.
\end{remark}

\begin{theorem}\label{th6.4} There is a version such that $\dis t\ra X_t$ is continuous from $[0,T]$ into $W_1^p(B_R)$.
\end{theorem}
\begin{proof} First for $b\in C_b^1(\R)$, by \eqref{eq1.6}, for $0<t_1<t_2<T$, we have
\begin{equation*}
X_{t_2}^\prime(x)=X_{t_1}^\prime(x)\, \exp\Bigl(\int_{t_1}^{t_2} b^\prime(X_s(x))\, ds\Bigr).
\end{equation*}
Then
\begin{equation}\label{eq6.7}
X_{t_2}^\prime(x) - X_{t_1}^\prime(x)=X_{t_1}^\prime(x)\,\Bigl[ \exp\Bigl(\int_{t_1}^{t_2} b^\prime(X_s(x))\, ds\Bigr) -1\Bigr].
\end{equation}
Using inequality $\dis |e^x-1|\leq |x|\, e^{|x|}\leq |x|\, (e^x+e^{-x})$, Relation \eqref{eq6.7} yields
\begin{equation*}
|X_{t_2}^\prime(x) - X_{t_1}^\prime(x)|\leq |X_{t_1}^\prime(x)|\Big(e^{\int_{t_1}^{t_2} b^\prime(X_s(x))\, ds}
+e^{-\int_{t_1}^{t_2} b^\prime(X_s(x))\, ds}\Bigr)\Bigl|\int_{t_1}^{t_2} b^\prime(X_s(x))\, ds\Bigr|.
\end{equation*}
Set
\begin{equation*}
A_{t_1,t_2}(x,w)=|X_{t_1}^\prime(x)|\Big(e^{\int_{t_1}^{t_2} b^\prime(X_s(x))\, ds}
+e^{-\int_{t_1}^{t_2} b^\prime(X_s(x))\, ds}\Bigr).
\end{equation*}
Using Theorem \ref{th6.2}, there exists a constant $C(T,||b||_\infty, q)$ such that for $t_1<t_2<T$ and $x\in \R$,
\begin{equation*}
\E\Bigl(A_{t_1,t_2}(x)^{2q}\Bigr)\leq C(T, ||b||_\infty,q).
\end{equation*}
On the other hand, by Proposition \ref{prop5.3}, there is a constant $\alpha>0$ such that
\begin{equation*}
\E\Bigl[ \Bigl(\int_{t_1}^{t_2} b^\prime(X_s(x))\, ds\Bigr)^{2q}\Bigr]\leq \alpha\, (t_2-t_1)^q.
\end{equation*}
Since
\begin{equation*}
|X_{t_2}^\prime(x) - X_{t_1}^\prime(x)|=A_{t_1,t_2}(x,w)\, \Bigl|\int_{t_1}^{t_2} b^\prime(X_s(x))\, ds\Bigr|,
\end{equation*}
by Cauchy-Schwarz inequality, we get for some constant $C>0$,
\begin{equation*}
\sup_{x\in \R}\E\bigl( |X_{t_2}^\prime(x)-X_{t_1}^\prime|^q\bigr)\leq C\, |t_2-t_1|^{q/2}.
\end{equation*}
Combining the estimate for $X_{t_2}(x)-X_{t_1}(x)$, we finally get

\begin{equation}\label{eq6.8}
\E\Bigl( ||X_{t_2}-X_{t_1}||_{W_1^p(B_R)}^q\Bigr)\leq C\, |t_2-t_1|^{q/2}.
\end{equation}
Finally for a bounded Borel function $b$, the relation \eqref{eq6.8} is true for $b_n$ with the constant $C$ independent of $n$, that is

\begin{equation*}\label{eq6.9}
\E\Bigl( ||\tilde X_{t_2}^n-\tilde X_{t_1}^n||_{W_1^p(B_R)}^q\Bigr)\leq C\, |t_2-t_1|^{q/2}.
\end{equation*}

According to convexity \eqref{eq6.4.2}, we finally get
\begin{equation*}\label{eq6.9}
\E\Bigl( ||X_{t_2}^n-X_{t_1}^n||_{W_1^p(B_R)}^q\Bigr)\leq C\, |t_2-t_1|^{q/2}.
\end{equation*}

Letting $n\ra +\infty$ and by lower semi-continuity of the norm, Relation \eqref{eq6.8} remains true for $b$ which is only bounded.
Now for $q>2$ and by Kolmogorov modification theorem, we get a version such $t\ra X_t$ is continuous from $[0,T]$ to $W_1^p(B_R)$.
\end{proof}

\section{Malliavin derivatives for SDEs with bounded drift}\label{sect7}

Again for $b\in C_b^1(\R)$, we  now make  the evidence of the dependence of $X_t$ with respect to random $w$
for the solution to SDE \eqref{eq1.1}:
\begin{equation*}
dX_t(w,x)=dW_t + b(X_t(w,x))\,dt,\quad X_0(w,x)=x.
\end{equation*}

In this part, we study the dependence $w\ra X_t(w,x)$. In other words, for fixed $x\in\R$ and $t$, we consider $X_t$ as a functional on
the Wiener space $\Omega=C_0([0,T])$, equipped with Wiener measure $\mu$. The Malliavin derivative for these functionals was considered   in \cite{Proske5}.
We refer to \cite{ImkellerRS} for new developments in this direction.

\vskip 2mm

Recall that $H$ is the Cameron-Martin space of $\Omega$; 
for $h\in H$,  $D_hX_t$ is the Malliavin derivative of $w\ra X_t(w,x)$ for $x$ fixed. Then $D_hX_t$ solves the following ODE

\begin{equation*}
d\bigl(D_hX_t\bigr)=\dot h(t)\,dt+ b^\prime(X_t)\, \bigl(D_hX_t\bigr)\, dt.
\end{equation*}
By Duhamel formula,
\begin{equation}\label{eq7.1}
D_hX_t=\int_0^t X_{t,s}^{\prime}(x)\dot h(s)\, ds,
\end{equation}

where $\dis X_{t,s}^\prime(x)=X_{t}^\prime(x)( X_{s}^\prime(x))^{-1}$ admits the expression, for $s<t$,
\begin{equation*}\label{eq7.2}
X_{t,s}^\prime(x)=\exp\Bigl(\int_s^t b^\prime(X_\theta(x))\, d\theta\Bigr).
\end{equation*}

By \eqref{eq7.1}, we have $\dis \frac{d}{d\tau}\bigl(\nabla X_t\bigr)(\tau)={\bf 1}_{(\tau <t)} X_{t,\tau}^\prime(x)$; therefore
\begin{equation}\label{eq7.3}
|\nabla X_t|_H^2=\int_0^t | X_{t,\tau}^\prime(x)|^2\, d\tau.
\end{equation}

\begin{theorem}\label{th7.1} Let $b$ be a bounded Borel function on $\R$ and $X_t$ solution to SDE \eqref{eq1.1}
for fixed $x$, then for any $p\geq 2$,
\begin{equation*}\label{eq7.4}
X_t\in \D_1^p(\Omega).
\end{equation*}
\end{theorem}

\begin{proof}
Consider the regularized sequence $b_n=b\ast \varphi_n$ in SDE \eqref{eq1.1} and $\tilde X_t^n$ the associated solution. Using \eqref{eq7.3}, we have
\begin{equation*}
|\nabla \tilde X_t^n|_H=\Bigl(\int_0^t | (\tilde X_{t,\tau}^n)^\prime(x)|^2\, d\tau\Bigr)^{1/2}.
\end{equation*}
For $p\geq 2$, by Jensen inequality, we have
\begin{equation*}
\E\bigl( |\nabla\tilde X_t^n|_H^p\bigr)
\leq C_{p,T}\int_0^T \E\bigl( |(\tilde X_{t,\tau}^n)^\prime(x)|^p\bigr)d\tau.
\end{equation*}

Since $||b_n||_\infty\leq ||b||_\infty$, by Theorem \ref{th6.2}, there is a constant $C(T,||b||_\infty)$ such that, for all $n$,
\begin{equation*}
\Bigl((\E\bigl( |\nabla \tilde X_t^n|_H^p\bigr)\Bigr)^{1/p}\leq C(T, ||b||_\infty)^{1/p}.
\end{equation*}
Let $X_t^n$ be defined in \eqref{eq6.4.2}. By convexity, we have
\begin{equation*}\label{eq7.5}
\Bigl((\E\bigl( |\nabla X_t^n|_H^p\bigr)\Bigr)^{1/p}\leq C(T, ||b||_\infty)^{1/p}.
\end{equation*}

This means that the sequence $\{X_t^n; n\geq 1\}$ is bounded in $\D_1^p(\Omega)$. Up to a subsequence, there exists $Y_t\in \D_1^p(\Omega)$ such that
$X_t^n$ converges to $Y_t$ weakly. By Banach-Saks theorem, up to a subsequence,
\begin{equation*}
\frac{X_t^1+\cdots + X_t^n}{n}\ra Y_t\quad \textup{in} \quad \D_1^p(\Omega).
\end{equation*}
 Since $X_t^n$ converges to $X_t$ in $L^p(\Omega)$ (see \cite{MeyerP}), we get $X_t=Y_t$; therefore $X_t\in \D_1^p(\Omega)$.
\end{proof}


\end{document}